\documentclass{article}

\usepackage{amsthm}  

\usepackage{amsmath}
\usepackage{latexsym}
\usepackage{amssymb}
\usepackage{verbatim}

\usepackage{tikz}
\usetikzlibrary{arrows,decorations.pathmorphing,backgrounds,positioning,fit}

\newtheorem{df}{Def}[section]  
\newtheorem{teo}[df]{Theorem}
\newtheorem{cor}[df]{Corollary}

\newtheorem{conv}[df]{Convention}

\newcommand{\eloise}{\exists \text{loise}}
\newcommand{\abelard}{\forall \text{belard}}

\author{Fausto Barbero}
\title{On existential declarations of independence in IF Logic}
\date{}

\begin{document}

\setcounter{page}{1}

\maketitle

\begin{abstract}
We analyze the behaviour of declarations of independence between existential quantifiers in quantifier prefixes of IF sentences; we give a syntactical criterion for deciding whether a sentence beginning with such prefix exists such that its truth values may be affected by removal of the declaration of independence. We extend the result also to equilibrium semantics values for undetermined IF sentences.

The main theorem allows us to describe the behaviour of various particular classes of quantifier prefixes, and to prove as a remarkable corollary that all existential IF sentences are equivalent to first-order sentences.

As a further consequence, we prove that the fragment of IF sentences with knowledge memory has only first-order expressive power.
\end{abstract}


\section*{Introduction and overview} 

Independence-Friendly (IF) logic is perhaps, among logics of imperfect information that have been introduced so far, the most suitable for expressing the game theoretical content of semantics (see \cite{HinSan89} and \cite{Hin96} for early presentations, and \cite{ManSanSev2011} for a more formal treatment). Its syntax extends common first-order logic introducing a new slash symbol that allows stating independence between quantifications; as an example, sentence
\[
\forall x\exists y\forall z(\exists w/\{x,y\})P(x,y,z,w)
\]
is true in the same models as the second-order existential sentence
\[
\exists f\exists g\forall x\forall zP(x,f(x),z,g(z))
\]
The intuitive reading of $(\exists w/\{x,y\})$ is ''there exists a w, independent of x and y''. We call $\{x,y\}$ the \emph{slash set of $w$}.

As first-order truth in a model can be characterized as the existence of a winning stategy (for a player that we shall call $\eloise$) in a certain semantic game, the same can be done for the new IF formulas by means of well-crafted games of imperfect information. Slash sets can be read as instructions for defining information sets in semantic games; once these have been introduced, the notion of strategy is restricted to functions that are constant over all information sets. 

It has been observed by Janssen (\cite{Jan2002}), however, that IF syntax allows expressing some declarations of independence which have no effect at all on the spectrum of truth values of a formula. The standard example is given by the sentence
\begin{equation}   \label{JAN}
\exists x (\exists y/\{x\})(x=y)
\end{equation}
which seems to assert that two agents may choose the same element from the domain being each unaware of the other's choice (an act which should be impossible over structures with two elements or more); yet, it is equivalent under IF semantics to 
\[  
\exists x\exists y(x=y)
\]
which is true in all structures.\footnote{We believe that the impossibility of formalizing correctly the intuition that underlies formula \eqref{JAN} is related to the difficulties that arise when we try to adapt standard notions of game theory to games of imperfect recall; see section \ref{GTS} for a brief presentation of this notion.}
What happens is that a winning strategy for \eqref{JAN} simply consists of two equal constants ($0$-ary functions). So, the game-theoretical structure of imperfect information has no effect on the truth values of \eqref{JAN}.

One may be tempted, now, to think that this kind of phenomenon arises anytime we have a declaration of independence between two existential quantifiers. The experience we have with first-order logic suggests that existentially quantified variables should never play the role of dependence variables. But it is known that this is not what happens in IF logic. Consider as an example\footnote{As exposed by Allen L.Mann at the \emph{Workshop on Dependence and Independence in Logic}, ESSLLI 2010.}
\begin{equation}   \label{REL}
\forall x\exists y(\exists z/\{x,y\})(x=z). 
\end{equation}
This formula is \emph{undetermined} on all structures with at least two elements. Instead, if we remove the declaration of independence between the existential quantifiers, we obtain Hodges' signalling formula (\cite{Hod97}) 
\[
\forall x\exists y(\exists z/\{x\})(x=z) 
\]
which is \emph{true} on all structures. (The value of $x$ may be copied in $y$, which in turn can be seen by $z$. We say that $y$ \emph{signals} the value of $x$ to $z$). 

As a third example, notice that formulas
\begin{equation}   \label{IRRWITHA}
\forall x(\exists y/\{x\})(\exists z/\{x,y\})(x=z) 
\end{equation}
and
\[
\forall x(\exists y/\{x\})(\exists z/\{x\})(x=z) 
\]
are strongly equivalent (they are both undetermined on all structures with at least two elements). This should subside the suspect that declarations of independence may be irrelevant only in existential formulas. What happens, here, is that the information generated by $\abelard$ for the universal quantifier is not accessible to $z$ even by means of signalling. 

The main purpose of the present paper will be that of classifying the behaviour of declarations of independence between existential quantifiers; we will provide a characterization theorem, of a game-theoretical yet syntactical nature, that will allow us to decide whether a declaration of independence behaves as in \eqref{JAN} and \eqref{IRRWITHA}, or instead as in \eqref{REL}.

To be more precise, we will not try to characterize what impact the declaration of independence has on the truth values of a fixed formula; we shall rather fix a quantifier sequence and investigate whether a sentence can be found which begins with such quantifier prefix and whose truth values are influenced by removal of the declaration of independence. This further degree of freedom will be needed in order to choose formulas which inhibit most signalling possibilities. 

In section \ref{IFL} we shall fix a syntax for IF logic and review its game-theoretical semantics and the game-theoretical terminology related with it. Section \ref{GTS} reviews briefly the notions of perfect recall, action recall and knowledge memory, which will be sometimes referred to, mainly for reasons of analogy. The main characterization result appears in section \ref{TCR}; some intermediate results, which do not take in account yet all signalling possibilities, are presented in section \ref{TOW}. In section \ref{SOMECON} some corollaries for particular classes of IF formulas are drawn from the main theorem.
Section \ref{EQU} extends the main result to the so-called \emph{equilibrium semantics} for IF logic, showing that probabilistic values for undetermined formulas behave as the main values \emph{true} and \emph{false}. 
In section \ref{SYNTREE} we generalize in a different direction, showing that the main result can be applied also to quantifier sequences which are not prenex. These considerations are not stated in full generality; rather, they are functional to the proof of the theorem which is presented in section \ref{KMIF}. There we prove that the fragment of IF logic with knowledge memory has first-order expressive power.

\section{IF logic} \label{IFL}
We mainly refer here to \cite{ManSanSev2011}, but our syntax slightly differs in that we do not allow function symbols in our formulas. This will help reduce the complexity of some proofs. Doing that, we are relying on the hypothesis that adding function symbols does not increase the expressive power of IF logic. (As far as we know, the problem has not been addressed in the current literature. We think that the situation should be analogous to that of first-order logic; but, since properties of IF logic often happened to defy intuitive considerations, we shall not make here any incautious guess).

Also, we shall restrict our attention to the so called \emph{regular formulas}; this is justified by the results of \cite{CaiDecJan2009} which show the strong equivalence of each IF formula with some regular formula. 
\begin{df}
A string $\varphi$ is an \textbf{IF formula} if it satisfies one of the following clauses: \\
1) $\varphi$ is a first-order atomic formula (without occurrences of function symbols).\\
2) $\varphi$ is $\sim\hspace{-4pt}\psi$, $\psi$ being a first-order atomic formula (without occurrences of function symbols).\\
3) $\varphi$ is $\psi \land \chi$ or $\psi \lor \chi$, $\psi$ and $\chi$ being IF formulas.\\
4) $\varphi$ is $(\exists v/V)\psi$ or $(\forall v/V)\psi$, $\psi$ being an IF formula, $v$ being a variable symbol and $V$ being a finite set of variable symbols (\emph{''slash set''}). We simply write $\exists v\psi$ or $\forall v\psi$ when the slash set is empty. 
\end{df}  

\begin{df}
The set of \textbf{free variables} of an IF formula $\varphi$, which we call $FV(\varphi)$, is defined recursively as:\\
1) $FV(R(x_1,\dots,x_n)) = \{x_1,\dots,x_n\}$ \\
2) $FV(\psi \land \chi) = FV(\psi \lor \chi) = FV(\psi) \cup FV(\chi)$ \\
3) $FV((\exists v/V)\psi) = FV((\forall v/V)\psi) = (FV(\psi)\setminus\{v\})\cup V$.
\end{df}

\begin{df}
An \textbf{IF sentence} is an IF formula $\varphi$ such that $FV(\varphi) = \emptyset$.
\end{df}

Sets $Subf(\varphi)$ of subformulas are defined as usual.

\begin{df}
An IF formula $\varphi$ is \textbf{regular} if the following hold:
1) If a quantification $(Qv/V)$ occurs in $\varphi$, and $x\in V$, then $(Qv/V)$ is subordinated to (i.e., it occurs within the syntactical scope of) another quantification, which is of the form $(Q'x/X)$\\
2) If a quantification $(Qv/V)$ occurs in $\varphi$, then it is not subordinated to any quantification of the form $(Qv/W)$. 
\end{df}

\begin{conv}
Everywhere in the present paper, we shall suppose that each variable is quantified only once in the formulas and in the quantifier prefixes we are speaking about. This is justified by the result, proved in \cite{CaiDecJan2009} that each IF formula is strongly equivalent (see below) to a regular formula, together with the fact that we are dealing mainly with prenex formulas. 
\end{conv}

We introduce now the game-theoretical semantics of IF sentences. Sensible semantics can be defined also for IF formulas (see \cite{Hod97} and \cite {Hod97b}, or also \cite{ManSanSev2011}), but we will not need it here.

We associate to each IF formula $\varphi$, each assignment $s$ such that $dom(s)\supseteq FV(\varphi)$ and each suitable structure $M$ a game $G(\varphi,s,M)$ of imperfect information (in case $s=\emptyset$, we may simply write $G(\varphi,M)$). Such game is defined by its set of players, by its rules and by its information sets. We will also introduce some extra terminology, which, although somewhat redundant, will allow us to speak more easily about the games. 

\begin{df} The set of players of $G(\varphi,s,M)$ is $N = \{\eloise,\abelard\}$. The following clauses are the \textbf{rules} of  $G(\varphi,s,M)$: 
\begin{enumerate}
\item At the beginning of the game, $\eloise$ plays as a Verifier, $\abelard$ plays as a Falsifier. 
\item If $\varphi$ is an atomic formula or a negated atomic formula, then the game ends, and, if $\varphi$ is true in $M$, then the Verifier wins; otherwise, the Falsifier wins.
\item If $\varphi=\psi\land\chi$, then the Falsifier chooses one out of $\psi$ and $\chi$, and game $G(\psi,s,M)$ (resp., $G(\chi,s,M)$) is played. 
\item If $\varphi=\psi\lor\chi$, then the Verifier chooses one out of $\psi$ and $\chi$, and game $G(\psi,s,M)$ (resp., $G(\chi,s,M)$) is played.
\item If $\varphi=(\forall v/V)\psi$, then the Falsifier chooses an element $c$ of $M$, and game $G(\psi,s(c/v),M)$ is played.
\item If $\varphi=(\exists v/V)\psi$, then the Verifier chooses an element $c$ of $M$, and game  $G(\psi,s(c/v),M)$ is played.
\end{enumerate}
\end{df}

\begin{df}
We can define more formally the structure of the game by defining its \textbf{histories} and the \textbf{associated assignments}. The set $H = \cup\{H_\psi | \psi\in Subf(\varphi)\}$ is defined inductively. \\
1) $(s,\varphi)\in H_\varphi$. \\
2) If $\psi\circ\chi\in Subf(\varphi)$, then $H_{\psi\circ\chi}=\{h^\frown \psi | h\in H_\psi\} \cup \{h^\frown \chi | h\in H_\chi\}$ (where $\circ$ is either $\land$ or $\lor$). \\
3) If $(Qv/V)\psi\in Subf(\varphi)$, then $H_{(Qv/V)\psi}=\{h^\frown(v,m)^\frown\psi | h\in H_\psi, m\in M\}$
 ($Q$ being a quantifier). \\
For each history $h\in H$ we define the associated assignment $s_h$ as:
\[
s_h = \left\{ \begin{array}{ll}
s 	&	\text{if } h = (s,\varphi) \\
s_{h'} &	\text{if } h = h'^\frown\psi\\
s_{h'}(v/a)  &   \text{if } h = h'^\frown(v,a)  \\
\end{array} \right.
\]

\end{df}

As usual in game theory, each history may identified with a \textbf{node} of a \textbf{game tree}. Histories which end with an atomic formula or a negation of an atomic formula are called \textbf{terminal histories}, or \textbf{plays}.

\begin{df}
Two histories $h$, $h'$ are in the same \textbf{information set}, and we write $h\sim h'$, if either holds:\\
1) $h$ and $h'$ are the same history \\
2) $h,h'\in H_{(Qv/V)\psi}$ and $s_h, s_{h'}$ only differ on variables of $V$.
\end{df}

Having defined the games, now we can define the semantics of IF logic.
\begin{df}
For each history $h$ we define
\[
P(h)= \left\{ \begin{array}{ll}
\eloise & \text{if } h\in H_{\psi \lor \chi} \text{ or } h\in H_{(\exists v/V)\psi} \\
\abelard & \text{if } h\in H_{\psi \land \chi} \text{ or } h\in H_{(\forall v/V)\psi} \\
\end{array}\right.
\]
(for some $\psi,\chi$) and we set $H_\exists = \{h\in H | P(h)=\eloise\}$ and $H_\forall = \{h\in H | P(h)=\abelard\}$.\\
A \textbf{(pure) strategy} for $\eloise$ is a set of functions $f_h$, one for each $h\in H_\exists$, such that:\\
1) $dom(f_h) = dom (s_h)$.\\
2) In case $h\in H_{\psi \lor \chi}$, $f_h(\vec x) = \psi$ or $f_h(\vec x) = \chi$.\\
3) In case $h\in H_{(\exists v/V)\psi}$, then $f_h(\vec x)\in M$, and $h\sim h'$ implies $f_h(\vec x) = f_{h'}(\vec x)$. \\ 
Strategies for $\abelard$ are defined analogously. \\
A \textbf{(pure strategy) profile} for an IF game is a couple $(\sigma,\tau)$ consisting of a strategy for $\eloise$ and one for $\abelard$.
\end{df}

Notice that each profile uniquely determines a play of the game, call it $z_{\sigma\tau}$. \\ 
Given a terminal history $h$, let $A_h$ be the last component of $h$ (i.e., the atomic or negated atomic formula on which the play ends). 

\begin{df}
The \textbf{utility function} for player $P$ is defined as:
\[
u_\exists(h) = \left\{ \begin{array}{ll}
1 & if M\models A_h   \\
0 & if M\not\models A_h \\
\end{array} \right.
\]
\[
u_\forall(h) = 1 - u_\exists(h).
\]
We say that a strategy $\sigma$ is \textbf{winning} for $\eloise$ if for any strategy $\tau$ of $\abelard$ we have $u_\exists(z_{\sigma\tau}) = 1$. \\ 
We define analogously winning strategies for $\abelard$.
\end{df}

\begin{df} (\textbf{Truth values})
An IF sentence $\varphi$ is \textbf{true in M} if $\eloise$ has a winning strategy for $G(\varphi,\emptyset,M)$. \\
It is \textbf{false in M} if $\abelard$ has a winning strategy for $G(\varphi,\emptyset,M)$. \\
Otherwise, we say that $\varphi$ is \textbf{undetermined in M}.
\end{df}

\begin{df}
We say that two IF sentences $\varphi,\varphi'$ are \textbf{truth} (resp. \textbf{falsity}) \textbf{equivalent}, and we write $\varphi\equiv\varphi'$, if they are true (resp. false) in the same structures.\\
We say $\varphi$ and $\varphi'$ are \textbf{strongly equivalent}, and we write $\varphi\equiv^*\varphi'$, if they are both truth and falsity equivalent.  
\end{df}

\section{A few words on game-theoretical structure} \label{GTS}
Now that we have associated an extensive game to each sentence and structure, we can look at the properties of such games. It turns out that some of these properties can be associated to the formula only, without reference to the particular structure under consideration; in some cases, they may be characterized by syntactical properties of the formula. A notable example is the property of \textbf{perfect recall}, which expresses the fact that a player never forgets his previous actions and knowledge. A game has perfect recall iff both of the following hold:
\begin{itemize}
\item \textbf{Action recall}: The player does not forget his own moves. 
\item \textbf{Knowledge memory}: The player does not forget his previous knowledge.
\end{itemize} 
It turns out that these two concepts are characterizable synctactically (\cite{Sev2006}); for IF formulas in negation normal form,
\begin{itemize}
\item $\eloise$ has action recall in $G(\varphi,M)$ iff in $\varphi$ there are no declarations of independence between existential quantifiers. 
\item $\eloise$ has knowledge memory in $G(\varphi,M)$ iff \\
 1) for every triple of existential quantifications that occur in $\varphi$, say $(\exists v/V)$ superordinated to $(\exists w/W)$ superordinated to $(\exists z/Z)$, $v\in Z$ implies $v\in W$, and \\
2) if a disjunction $\lor_i$ occurs superordinated to $(\exists z/Z)$, then $v\in Z$ implies that $v$ is not quantified superordinated to $\lor_i$.
\end{itemize}
(Clumsier characterizations can be given for generic IF formulas).\\ 
Thus, perfect recall itself is characterizable by synctactical means. So, we can speak of ''formulas of perfect recall''. It turned out that IF formulas are quite essentially formulas of \emph{imperfect} recall, meaning that all IF formulas of perfect recall are truth  equivalent to some first-order formula, and also falsity equivalent to some (possibly different) first-order formula (see \cite{Sev2006} or \cite{ManSanSev2011} for a proof).

\section{Towards a game-theoretical characterization} \label{TOW}
Here we consider a weak notion of relevance of declarations of independence.
\begin{df}
Given a formula $\varphi$ with a single occurrence of a quantifier $(\exists y/\dots x$ $\dots)$, we denote 
as $\varphi^{y\leftarrow x}$ the formula obtained from $\varphi$ by removing the declaration of independence from $x$.
(In the present paper we are only interested in the case that $x$ occurs in an $\emph{existential}$ superordinated quantification).
\end{df}

\begin{df}
A declaration of independence in an existential quantifier from an existential quantifier ($I\exists\exists$) 
occurring in a quantifier prefix $\vec Q$ is \textbf{relevant} if there exist a sentence $\varphi=\vec Q\psi$, and a structure $M$, such that $\varphi$ and $\varphi^{y\leftarrow x}$ have different truth values on $M$. Otherwise, it is said to be \textbf{irrelevant}.
\end{df}

Clearly, any formula containing an $I\exists\exists$ lacks action recall. We claim that relevance of the 
$I\exists\exists$ is tightly connected to the structure of knowledge memory of the formula. 

A remark: we focus on this case, and not on declarations of indipendence of existential quantifiers from universal quantifiers ($I\forall\exists$) because the latter have trivial behaviour in this context. Indeed, formula
\[
\dots(\forall x/X)\dots(\exists y/Y)\dots(x_1=0\land\dots\land x_n=0\land x=y)
\] 
(where $x\in Y$ and the $x_i$s are all the existential variables occurring in the quantifier prefix) is not true, while
\[
\dots(\forall x/X)\dots(\exists y/(Y\setminus\{x\}))\dots(x_1=0\land\dots\land x_n=0\land x=y)
\]
is. Declaration types $I\exists\forall$ and $I\forall\forall$ may be treated dually.

In the present paragraph, we will define two conditions related to knowledge memory that will be proved to be, respectively, a sufficient and a necessary condition for relevance. This will help the reader to get acquainted, in a simpler context, with some proof methods that will be applied  in the next paragraph in order to prove a characterization result. This is our candidate as a sufficient condition:  
\begin{df}
We say that an $I\exists\exists$ (say $(\exists y/Y)$, subordinated to $(\exists x/X)$ and such that $ x\in Y$), occurring in a formula $\varphi$ \textbf{breaks knowledge memory} if there is a universally quantified variable $v$, superordinated to $(\exists x/X)$, such that $v\in Y\setminus X$. 
\end{df} 
Clearly, if an $I\exists\exists$ breaks knowledge memory for a certain prenex formula $\phi$ of quantifier prefix $\vec Q$, it breaks knowledge memory for any other prenex formula with the same quantifier prefix; we may say, thus, that it breaks knowledge memory \textit{in} $\vec Q$. \\
Notice also that this is a more specific notion than ''lacking knowledge memory''; here we make reference to a specific superordinate quantifier $\exists x$. \\ 
Furthermore, specific reference to universal quantifications is needed because it is only at universal quantifications that information is generated that could happen not be available to $\exists$loise (this intuitive claim is proved in the present paper  for any quantifier prefix; of course it does not hold if we also have connectives superordinated to quantifiers).  
If we only requested that $Y\not\subset X$, there would be purely existential formulas whose quantifiers break knowledge memory and yet are not relevant. Such an example is Janssen's formula \eqref{JAN}. However, we give a name to such a condition for later usage:  
\begin{df}
We say that an $I\exists\exists$ (say $(\exists y/Y)$, subordinated to $(\exists x/X)$ and such that $ x\in Y$), occurring in a formula $\varphi$ \textbf{weakly breaks knowledge memory} if there is a variable $v$, superordinated to $(\exists x/X)$, such that $v\in Y\setminus X$.
\end{df}

Now we state and prove the sufficient condition for relevance.

\begin{teo}   \label{SUFF}
If an $I\exists\exists$ occurring in a quantifier prefix $\vec Q$  breaks knowledge memory in $\vec Q$, then it is relevant in $\vec Q$.
\end{teo}

\begin{proof} 
Suppose the prefix $\vec Q$ is of the form $\dots(\exists x/X)\dots(\exists y/Y)\dots$, where $x\in Y$. Now fix a universally quantified variable $v\in Y\setminus X$ which occurs superordinated to $(\exists x/X)$. There is at least one because our $I\exists\exists$ breaks knowledge memory. Denote the remaining 
existentially quantified variables as $x_1,\dots,x_n$. 
Let $W$ be the set of all variables occurring in $\vec Q$; let $w$ be a new variable. 
Define a sentence\footnote{Notice that this sentence imposes no constraints on universally quantified variables except for $v$. So, it cannot happen that $\abelard$ is forced to signal the value of $v$ through these variables.} 
\[
\varphi=\vec Q(\exists w/W)(v=y \land x_1=w \land \dots \land x_n=w),
\]

Let $|M|$ be a set containing at least two distinct elements. Let $M$ be any structure with domain $|M|$.

We prove that there is no winning strategy for $\exists$loise in $G(\varphi,M)$. Suppose $\sigma$ is a winning strategy. \\
Let $h$ be a terminal history that can be played as $\eloise$ plays according to $\sigma$ (remember that, once a strategy for $\eloise$ is fixed, the actual play is determined by $\abelard$'s moves only). Suppose that in $h$, for some $i$, different objects are chosen for $x_i$ and for $w$. Then, there is another terminal history $h'$ that differs from $h$ at most in that $\abelard$ chooses ''$x_i=w$'' as a last move. Since this history is also playable as $\eloise$ plays according to her supposedly winning strategy $\sigma$, and in it $\abelard$ wins, we have a contradiction.

Suppose now that, in $h$, the same value is chosen for all $x_1,\dots, x_n$ and $w$. Then, there is another terminal history $g$ that can be played as $\eloise$ plays according to $\sigma$ and which differs from $h$ at most in that $\abelard$, as a last move, chooses the leftmost conjunct $v=y$.

Suppose that $\abelard$'s choice for $v$ in $g$ is some element $e$. Since there are at least two distinct elements in $|M|$, there is another terminal history $g'$, playable according to $\sigma$, in which $\abelard$'s moves differ only in that he  chooses something else for $v$ (say $e'\neq e$). It may happen that, as a consequence, now $\sigma$ prescribes different choices for $w$ and for some $x_i$; if that happens, we are again in the former case, and we obtain a contradiction. Suppose instead that $\sigma$ prescribes the same choices for all $w,x_1,\dots,x_n$ also in $g'$. Let $A$ be the node of choice for $y$ which is reached in $g$, and let $B$ be node of choice for $y$ which is reached in $g'$. $A$ and $B$ are in the same information state, since the histories which end with them differ only in the choice for $v\in Y$. So, $\sigma$ prescribes the same move in both $A$ and $B$. But there is no choice for $y$ that leads to a victory for $\eloise$ in both $g$ and $g'$. So, $\sigma$ is not a winning strategy. 

Instead, there is at least one winning strategy for $\exists$loise in $G(\varphi^{y\leftarrow x}, M)$:  for variable $x$, choose the same value that $\abelard$ has chosen for variable $v$. This is allowed because $x$ depends on $v$.

Choose for $y$ the same value that was chosen for $x$; this is allowed because $y$ depends on $x$. 

Choose the same fixed value for all variables $x_1,\dots,x_n,w$. \\
\end{proof}

The reader may have noticed that the proof procedure dictates only very generic requirements concerning the construction procedure for the model; it is only asked that the domain contains at least two elements. We will discuss this point in the next section (corollary \ref{MODELRELEV}).\\

Here we add a necessary condition for relevance. The result is not so interesting in itself, but it suggests what we should do in order to find a characterization of relevance; the proof illustrates how signalling works in simple cases.

\begin{teo}   \label{NEC}
If an $I\exists\exists$ occurring in a quantifier prefix $\vec Q$  is relevant in $\vec Q$, then it weakly breaks knowledge memory in $\vec Q$.
\end{teo}
\begin{proof}
Suppose that the $I\exists\exists$ does not weakly break knowledge memory. Suppose $\varphi$ is a sentence beginning with prefix $\vec Q$, and suppose we have a winning strategy for $\exists$loise in $\varphi^{y\leftarrow x}$; in it,
$\exists$loise chooses $x$ according to a function $f_x(u_1,\dots,u_s)$ and $y$ according to a function 
$f_y(v_1,\dots v_t,x)$ such that $\{u_1,\dots,u_s\}\subset\{v_1,\dots v_t\}$. This inclusion holds because the $I\exists\exists$ does not weakly break knowledge memory. We may turn $\sigma$ into a winning 
strategy for $\varphi$ by choosing $y$ according to a strategy function $g_y(v_1,\dots,v_t)$
defined using the (permitted) value of $f_x(u_1,\dots,u_s)$ anywhere $x$ occurred in the definition of 
$f_y(v_1,\dots,v_t,x)$.\\
Other cases: \\ - Suppose $\exists$loise has no winning strategy for $\varphi^{y\leftarrow x}$; then she cannot have one for $\varphi$,
whose game is harder to win due to lack of information.\\ 
-Suppose $\forall$belard has a winning strategy for $\varphi^{y\leftarrow x}$; then, this is also winning  
for $\varphi$, because slashes in existential quantifications play no role for $\forall$belard \\
-Suppose $\forall$belard has no winning strategies for $\varphi^{y\leftarrow x}$: then, for the same reason, he has none
for $\varphi$.
\end{proof}

\section{The characterization result} \label{TCR}

We have seen that (Theorem \ref{NEC}) if an existential declaration of independence of $y$ from $x$ is relevant, then there must be some information which is available to the verifier in $x$ but is explicitly hidden to the verifier in $y$. Also, it seems likely that the information we are speaking about is generated by $\abelard$'s moves, i.e. by the choice made for some universal quantification. Lemma \ref{SUFF} agrees with this observation. All of this suggests a generalization of the notions introduced in the previous paragraph.

\begin{df}
Given two quantifications $(Qv/V)$ and $(Q'v'/V')$ occurring in a quantifier prefix $\vec P$, we write 
\[
(Qv/V) \prec (Q'v'/V')
\]
 in order to say that $(Qv/V)$ is superordinated to $(Q'v'/V)$ in $\vec P$.
\end{df} 
\begin{df} \label{DEFBSS}
A sequence
\[
 ((\forall v_k/V_k),(\exists v_{k-1}/V_{k-1}),\dots,(\exists v_1/V_1),(\exists x/X),(\exists y/Y))
\]
 of quantifications is a \textbf{broken signalling sequence} if the following conditions hold:\\
1)   $x\in Y$ \\
2)   $(\forall v_k/V_k) \prec (\exists v_{k-1}/V_{k-1})\prec\dots\prec(\exists v_1/V_1)\prec(\exists x/X)\prec(\exists y/Y)$ \\
3)   $v_1\notin X$ and $v_i\notin V_{i-1}$, for all $i = 2..k$ \\
4)   $v_1,\dots,v_k\in Y$     \\
In case $k=1$, we always request that $v_k$ is universally quantified.
\end{df}

\begin{teo}  \label{THEONE}
Suppose there is an $I\exists\exists$  of the form $(\exists y/\dots x\dots)$ occurring in a quantifier prefix $\vec Q$. Let $Y$ be be a name for the slash set of $y$. Then, the $I\exists\exists$ is relevant if and only if $\vec Q$ contains a broken signalling sequence of the form $((\forall v_k/V_k),(\exists v_{k-1}/V_{k-1}),\dots,(\exists v_1/V_1),(\exists x/X),(\exists y/Y))$.
\end{teo} 
\begin{proof} 

$\Longleftarrow$) This proof is very similar to that of Theorem \ref{SUFF}; variable $v_k$ plays the same role that was played before by variable $v$. We just explain what modifications should be made.
Fix a broken signalling sequence in $\vec Q$. Let $x_1,\dots,x_n$ be the \emph{remaining} existentially quantified variables of $\vec Q$, i.e., those that are different from $v_{k-1},\dots,v_1,x,y$; let $W$ be as before. 
Let $\varphi$ be: 
\[
\varphi=\vec Q(\exists w/W)(v_k=y \land x_1=w \land \dots \land x_n=w).
\]
The model $M$ is defined as before, as any structure whose domain contains at least two elements.

From the hypothesis that $\eloise$ has a winning strategy for $G(\varphi,M)$ we get a contradiction as in the proof of Theorem \ref{SUFF}. The contradiction derives from conditions 1 and 4.
 Indeed, define $h$, $g$, $A$ and $B$ similarly as in the proof of Theorem \ref{SUFF}. The only difference is that we cannot guarantee the existence of a $g$ that differs from $h$ only for what regards the choice of a value for $v_k$. Changing the value of $v_k$ will make so that strategy $\sigma$ will assign new values to $v_{k-1},\dots,v_1$. But, also in this case, conditions 1 and 4 guarantee that $\sigma$ prescribes the same move for both $A$ and $B$. So we get the contradiction as before.

$\eloise$ has a winning strategy for $G(\varphi^{y\leftarrow x},M)$.  We describe it.\\
For $i=1..k-1$ choose for variable $v_i$ the same value that was chosen for variable $v_{i+1}$. This is allowed by condition 3 in the definition of broken signalling sequence.  \\
For variable $x$, choose the same value that was chosen for $v_1$. Again, we are relying on condition 3.\\
For $y$, choose the same value that was chosen for $x$.\\
Choose one  and the same value for $x_1,\dots,x_n,w$. 

$\Longrightarrow$) Suppose that in $\vec Q$ there are no broken signalling sequences ending with $(\exists x/X),(\exists y/Y)$. Let $\varphi$ be a sentence which begins with prefix $\vec Q$, and $M$ a structure. \\
As a first case, suppose that $\eloise$ has a winning strategy $\sigma$ for $G(\varphi^{y\leftarrow x},M)$. We want to produce a winning strategy for $\eloise$ in $G(\varphi,M)$. We will do so through an iterative process of elimination of forbidden variables. In order to keep track of what happens during this process, we shall make use of an indexed notation for denoting variables. The index will be a \emph{sequence} of numbers. For example, variable $u_{(n_1,\dots n_l)}$ will be the variable that is met during the $l$-th stage of the iterative process; to be more precise, it will be the $n_l$th among the existentially quantified variables which are associated to variable $u_{(n_1,\dots n_{l-1})}$ (which has been, in turn, generated during the $n_{l-1}$th stage, and so on). Instead, variable $z_{(n_1,\dots n_l)}$ is a universally quantified variable which is met during the $l$-th stage of the iterative process, and, as $u_{(n_1,\dots n_l)}$, is associated to variable $u_{(n_1,\dots n_{l-1})}$. 

Now, strategy $\sigma$ chooses $x$ according to a function $f_x(u_{(1)},\dots,u_{(m)},z_{(1)},$ $\dots,z_{(n)})$, where $u_{(1)},\dots,u_{(m)}$ are existentially quantified variables, and $z_{(1)},$ $\dots,z_{(n)}$ are universally quantified variables. And it chooses $y$ according to some function $f_y(v_1,\dots,v_t,x)$. We want to replace the reference to $x$ with a functional expression which contains, as variables, only variables which are not in $Y$, i.e. variables that are visible for $y$ in $\varphi$. \\
As a stage 0, replace $x$ in $f_y(v_1,\dots,v_t,x)$ with $f_x(u_{(1)},\dots,u_{(m)},z_{(1)},\dots,$ $z_{(n)})$; call $t_0$ the object thus obtained. \\
As stage 1, let $C_1$ be the set of variables $u_{(i)}$ such that either (A) $u_{(i)}$ does not depend on any existentially quantified variable or (B) $u_{(i)}\notin Y$; in case (B) does not hold, replace such variables in $t_0$ with expressions $f_{u_{(i)}}(z_{(i,1)},z_{(i,2)},\dots)$ (in which the list of arguments could also be empty) where $f_{u_{(i)}}$ is the function which $\sigma$ assigns to variable $u_{(i)}$. Let, also, $D_1$ be the set of variables $u_{(i)}$ which are not in $C_1$. These ones are to be replaced in $t_0$ by expressions $f_{u_{(i)}}(u_{(i,1)},u_{(i,2)},\dots,z_{(i,1)},z_{(i,2)},\dots)$, as established by $\sigma$. The final result of these substitutions will be an expression $t_1$.\\
At stage 2, let $C_2:=\{ u_{(i,j)} | u_{(i)}\in D_1 \text{ and } u_{(i,j)} \text{ satisfies either (A) or (B)}\}$ and $D_2:=\{ u_{(i,j)} | u_{(i)}\in D_1 \text{ and } u_{(i,j)}\notin C_2\}$. Perform in $t_1$ analogous substitutions as those that were made at stage 1. \\
The generic $n$-th stage is performed analogously to stage 2. The process ends as soon as some $D_r$ is empty. The correspondent expression $t_{r-1}$ will be a nested application of functions; the only unnested objects occurring in it will be nullary functions (which correspond to existentially quantified variables of some $C_i$) and variables.\\
Now, $\eloise$ is allowed to use any of her strategy functions in order to define a strategy function for choosing $y$; furthermore, the existentially quantified variables that are left in $t_{r-1}$ are those that satisfy (B), so they are accessible for $y$.  We want to show that $\eloise$ is also allowed to use any of the universally quantified variables occurring in $t_{r-1}$. If we manage to do so, we will be able to define a winning strategy for $\eloise$ in $G(\varphi,M)$ just by replacing, in $\sigma$, function $f_y$ with the function defined by the expression $t_{r-1}$, and so our theorem will be proved. \\
So, suppose that some universally quantified variable $z_{(n_1,\dots,n_s)}$ occurring in $t_{r-1}$ is not accessible to $y$, i.e. suppose that $z_{(n_1,\dots,n_s)}\in Y$. We will show that, then, the sequence $\vec S = ((\forall z_{(n_1,\dots,n_s)}/Z_{(n_1,\dots,n_s)}),(\exists u_{(n_1,\dots,n_{s-1})}/U_{(n_1,\dots,n_{s-1})}),$ $\dots,(\exists u_{(n_1)}/U_{(n_1)}),(\exists x/X),(\exists y/Y))$ is a broken signalling sequence: a contradiction. \\
Observe that $u_{(n_1,\dots,n_{s-1})}$ does not satisfy (B), i.e., it holds that $u_{(n_1,\dots,n_{s-1})}\in Y$. Analogously, we can prove that $u_{(n_1,\dots,n_{s-2})},\dots, $ $u_{(n_1,n_2)},u_{(n_1)}\in Y$. So, $\vec S$ satisfies condition 4 of the definition of broken signalling sequence. We also assumed that $x\in Y$: so, condition 1 is satisfied. Conditions 2 and 3 are trivially satisfied by construction. \\
This reasoning accounts for the case in which $\eloise$ has a winning strategy for $G(\varphi^{y\leftarrow x},M)$. The remaining cases are dealt with as in the proof of theorem \ref{NEC}.  
\end{proof}

Since the proof of the right-to-left implication defines a very generic model, we may strenghten the consequences of there being a broken signalling sequence: 

\begin{cor} \label{MODELRELEV}
Let  $\vec Q$ be a quantifier prefix. Suppose it contains a broken signalling sequence $((\forall v_k/V_k),(\exists v_{k-1}/V_{k-1}),\dots,(\exists v_1/V_1),(\exists x/X),(\exists y/Y))$. Then, there is an IF sentence $\varphi = \vec Q\psi$ such that, for all structures $M$ containing at least two elements, $\varphi^{y\leftarrow x}$ is true in $M$, while $\varphi$ is not.
\end{cor}

\section{Some consequences}  \label{SOMECON}
We can see that the main theorem has immediate consequences for simpler classes of sentences.

\begin{cor} \label{COR1}
Let $\varphi$ be a prenex existential IF sentence. We may remove or add any superordinated variable in the slash sets of $\varphi$, obtaining still a strongly equivalent formula.
\end{cor}

\begin{proof}
Let $\varphi'$ be obtained from $\varphi$ by adding to each slash set all of its superordinated variables.
Since there are only existential quantifiers in $\varphi'$, all of its declarations of independence are $I\exists\exists$s. And, since there are no universal quantifiers occurring in $\varphi'$, its quantifier prefix cannot contain any broken signalling sequence. So, no declaration of independence in $\varphi'$ is relevant. Thus, they may be removed from slash sets without changing the truth value of the formula. So, all of the formulas that can be obtained from $\varphi'$ by removing some declarations of independence (including, thus, $\varphi$) are strongly equivalent. 
\end{proof}

This implies, in particular:

\begin{cor}
Every prenex existential IF sentence is strongly equivalent to a first-order sentence.
\end{cor}

\begin{proof}
Just remove all declarations of independence from the sentence; the resulting formula will be strongly equivalent (by \ref{COR1}) and will have first-order syntax. It will also be true in the same models in IF semantics as in first-order semantics, since IF logic is a conservative extension of first-order logic; it will be false in the same models because first-order formulas satisfy the excluded middle principle.
\end{proof}

(this fact can be given a much simpler proof if we make use of Hodges' compositional semantics, as is suggested in \cite{DurKon2011}).\\
More generally, by similar arguments we can prove:

\begin{cor}
Let $\vec Q$ be a quantifier prefix which contains only existential quantifications. Let $\vec Q\psi$ be an IF sentence. Let $\vec{Q'}$ be obtained from $\vec Q$ by removing or adding arbitrary superordinated variables from the slash sets. Then, 
\[
 \vec Q\psi \equiv^* \vec{Q'}\psi  
\]
In particular, if $\psi$ is a first-order formula, then $\vec Q\psi$ is strongly equivalent to a first-order sentence. 
\end{cor}

As a last generalization,

\begin{cor}
Let $\vec Q$ be a quantifier prefix of the form $\exists\forall\exists^1$ (i.e. a sequence of existential quantifications, followed by a sequence of universal quantifications, followed by one single occurrence of an existential quantification). Let $\vec Q\psi$ be an IF sentence. Let $\vec{Q'}$ be obtained from $\vec Q$ by removing or adding arbitrary superordinated variables from the slash sets of existential quantifiers. Then, 
\[
 \vec Q\psi \equiv^* \vec{Q'}\psi  
\]
In particular, if $\psi$ is a first-order formula, then $\vec Q\psi$ is truth equivalent to a first-order sentence. 
\end{cor}
This last result only guarantees truth-equivalence: our main theorem does not allow us to remove declarations of independence from universally quantified variables.\\

Besides classifying quantifier prefixes according to the kind of quantifications involved, we can look instead at the kind of independence sets that appear in them. From this point of view, we also have that\footnote{The following corollary was suggested by Allen L. Mann, personal communication.}: 

\begin{cor}
Let $\vec Q$ be a quantifier prefix whose independence sets contain only existentially quantified variables. Let $\vec Q\psi$ be an IF sentence. Let $\vec{Q'}$ be obtained from $\vec Q$ by removing arbitrarily variables from the slash sets of existential quantifiers, or also adding them existentially quantified variables. Then, 
\[
 \vec Q\psi \equiv^* \vec{Q'}\psi  
\]
In particular, if $\psi$ is a quantifier-free formula, then $\vec Q\psi$ is strongly equivalent to a first-order formula. 
\end{cor}
\begin{proof}
Any broken signalling sequence contains at least one declaration of independence from a universally quantified variable (what we called $v_k$). So our quantifier prefix does not contain any broken signalling sequence, and we may apply Theorem \ref{THEONE}. 
\end{proof}

\section{Relevance and equilibrium semantics}  \label{EQU}
Our main theorem \ref{THEONE} regards only the canonical truth values \emph{true, false, undetermined}. But, IF semantics can also be extended to the so-called \emph{equilibrium semantics}, which assigns a sort of probabilistic values to undetermined formulas, thus providing a finer grain classification of IF sentences (equilibrium semantics were introduced in \cite{Sev2006} and then developed in \cite{SevSan2010} and \cite{ManSanSev2011}). We will briefly review these notions and then show that the new values behave like the canonical truth calues as regards relevance.

\begin{df}
A \textbf{mixed strategy} for player $P$ is a probability distribution over the set of strategies of $P$. A \textbf{(mixed strategy) profile} for an IF game is a couple $(\mu,\nu)$ consisting of a mixed strategy for $\eloise$ and one for $\abelard$.
The \textbf{expected utility} for player $P$ and profile $(\mu,\nu)$ is defined as
\[
U_P(\mu,\nu) = \sum_{(\sigma,\tau)\in Pr}\mu(\sigma)\nu(\tau)u_P(z_{\sigma\tau}) 
\]
where we denote as $Pr$ the set of pure strategy profiles for the game under consideration.
\end{df}
\begin{df}
An \textbf{equilibrium} for a game $G(\varphi, s ,M)$ is a mixed strategy profile $(\mu,\nu)$ such that, for every mixed strategy $\zeta$ of $\eloise$, and for every mixed strategy $\xi$ of $\abelard$
\[
U_\exists(\zeta,\nu) \leq U_\exists(\mu,\nu)
\]
\[
U_\forall(\mu,\xi) \leq U_\forall(\mu,\nu)
\]
(i.e., no player can improve his expected utility by means of unilateral deviation).
\end{df}

Since IF games are constant sum games ($u_\exists(h)+u_\forall(h) = 1$ for all terminal histories), we can rewrite the conditions for equilibrium as:
\[
U_\exists(\zeta,\nu) \leq U_\exists(\mu,\nu) \leq U_\exists(\mu,\xi).
\]
The fact that the game is constant sum also implies that $U_\exists(\mu,\nu) = U_\exists(\mu',\nu')$ for any pair of equilibria $(\mu,\nu),(\mu',\nu')$; this justifies the following definition.

\begin{df}
Let $\varphi$ be an IF sentence, $M$ a structure. Suppose $G(\varphi,\emptyset,M)$ has an equilibrium $(\mu,\nu)$. Then we define the \textbf{value} of the sentence on $M$ as $v(\varphi,M) = U_\exists(\mu,\nu)$.
\end{df}

\begin{teo} (\cite{Sev2006})
An IF sentence $\varphi$ is true on $M$ iff $v(\varphi,M) = 1$; it is false iff $v(\varphi,M) = 0$. 
\end{teo}

So, we can say that these values define an extension of IF semantics (''equilibrium semantics''). Undetermined formulas receive values from the open interval $(0,1)$.

\begin{df}
An $I\exists\exists$ occurring in a quantifier prefix $\vec Q$ is \textbf{$\epsilon$-relevant} if there exist a sentence $\varphi=\vec Q \psi$, and a structure $M$, such that either $v(\varphi,M)\neq v(\varphi^{y\leftarrow x},M)$, or one of the values is defined and the other is not.
\end{df}

Clearly, relevance implies $\epsilon$-relevance. The opposite implication, which is not quite obvious, follows from the next theorem.

\begin{teo}
Suppose there is an $I\exists\exists$  of the form $(\exists y/\dots x\dots)$ occurring in a quantifier prefix $\vec Q$. Let $Y$ be be a name for the slash set of $y$. Then, the $I\exists\exists$ is $\epsilon$-relevant if and only if $\vec Q$ contains a broken signalling sequence of the form $((\forall v_k/V_k),(\exists v_{k-1}/V_{k-1}),\dots,(\exists v_1/V_1),(\exists x/X),(\exists y/Y))$.
\end{teo}

\begin{proof}
$\Longleftarrow$) The hypothesis implies relevance, by Theorem \ref{THEONE}, and thus $\epsilon$-relevance.\\
$\Longrightarrow$) We shall refer, in what follows, to the utility function for player $P$ in the game $G(\varphi^{y\leftarrow x},M)$ as $u_P^1$; for game $G(\varphi,M)$, we write $u_P^2$. We follow an analogous convention for expected utilities.

Suppose that in $\vec Q$ there are no broken signalling sequences ending with $((\exists x/X),(\exists y/Y))$. 
Thanks to this hypothesis, we are able, by means of the procedure described in the proof of Theorem \ref{THEONE}, to associate to any strategy function $f$ for $y$ from a game $G(\varphi^{y\leftarrow x},M)$ a term $t_f$ which defines $f$ referring only to variables which are not in $Y$. Then, to any pure strategy $p$ for $\eloise$ in $G(\varphi^{y\leftarrow x},M)$ we may associate another pure strategy $p'$ by replacing the choice function for $y$ with its associate term. This is a correct definition of a strategy for $G(\varphi,M)$. Notice 1) that $u_\exists^2(p',q) = u_\exists^1(p,q)$ against any strategy $q$, and 2) that the operator ' is surjective on the set of strategies for $G(\varphi,M)$ (indeed, $Strat(\varphi,M) \subseteq Strat(\varphi^{y\leftarrow x},M)$, and the operator fixes all strategies in $Strat(\varphi,M)$ -- the procedure never removes variables which are not in $Y$). To any mixed strategy $\sigma$ for $\eloise$ in $G(\varphi^{y\leftarrow x},M)$ we now associate a mixed strategy $\sigma'$ for $G(\varphi,M)$ defined as  $\sigma'(p') = \sum_{q\in\Xi_{p'}} \sigma(q)$, where $\Xi_{p'} = \{q\in Strat(\varphi^{y\leftarrow x},M) | q'=p'\}$ (clearly, $\sigma'$ is well-defined).
Again we have, for expected utilities, that $EU_\exists^2(\sigma',\tau)=EU_\exists^1(\sigma,\tau)$ against any $\tau$.

Now, suppose $\varphi$ and $\varphi^{y\leftarrow x}$ are both undetermined in a structure $M$ (the other cases are already taken care of by Theorem \ref{THEONE}). To any strategy profile $(\sigma,\tau)$ for $G(\varphi^{y\leftarrow x},M)$ we associate a profile $(\sigma',\tau)$ for $G(\varphi,M)$ as described above. We want to calculate the expected utilities for each player $\rho$. Let us now define $S=\{(p,q)|(p,q)\text{ is a strategy profile for } G(\varphi,M) \}$ and
 $X=\{(p,q)|(p,q)\text{ is a }$ 
$ \text{strategy profile for } G(\varphi^{y\leftarrow x},M) \}$. We have:
\begin{eqnarray} \nonumber
EU_\rho^2(\sigma',\tau) & = & \sum_{(p,q)\in S} \sigma'(p)\tau(q)u_\rho^2(p,q)          \nonumber    \\
                                 & = & \sum_{(p,q)\in S}\left[ \left(\sum_{r\in\Xi_p}\sigma(r)\right)\tau(q)u_\rho^2(p,q)\right]         \nonumber     \\
                                 & = & \sum_{(p,q)\in S}\sum_{r\in\Xi_p}\sigma(r)\tau(q)u_\rho^1(r,q)           \nonumber   \\
	 		 & = & \sum_{(r,q)\in X} \sigma(r)\tau(q)u_\rho^1(r,q)            \nonumber  \\
			 & = & EU_\rho^1(\sigma,\tau)					 \nonumber
\end{eqnarray}
In the second line we used the definition of $\sigma'$. In the third line we used distributivity, together with the fact that $u_\rho^2(p,q) = u_\rho^1(r,q)$. \\
Now, since the main equality holds for any profile $(\sigma, \tau)$ of game $G(\varphi^{y\leftarrow x},M)$, and $Strat(\varphi,M) \subseteq Strat(\varphi^{y\leftarrow x},M) $, we get that if $(\sigma,\tau)$ in an equilibrium for $\varphi^{y\leftarrow x}$, then $(\sigma',\tau)$ is an equilibrium for $\varphi$ with the same value. Suppose indeed there is a strategy $\tau^*$ for $\abelard$ such that $EU_\exists^2(\sigma',\tau^*) < EU_\exists^2(\sigma',\tau)$. Then, by the equality,  $EU_\exists^1(\sigma,\tau^*) < EU_\exists^1(\sigma,\tau)$, contradicting the fact that $(\sigma,\tau)$ is an equilibrium of $G(\varphi^{y \leftarrow x},M)$. Analogously (using also $Strat(\varphi,M) \subseteq Strat(\varphi^{y\leftarrow x},M)$) we can prove that inequalities of the form $EU_\exists^2(\sigma^*,\tau) > EU_\exists^2(\sigma',\tau)$ cannot hold.

We want to prove, now, that any equilibrium $(\sigma,\tau)$ for $G(\varphi,M)$ is also an equilibrium for $G(\varphi^{y \leftarrow x},M)$. Suppose instead that there is some strategy $\sigma^*$ for $\eloise$ in $G(\varphi^{y\leftarrow x},M)$ such that $EU_\exists^1(\sigma^*,\tau) > EU_\exists^1(\sigma,\tau)$. So, thanks again to the main equality, and since $\sigma'=\sigma$ (because operator ' fixes strategies of $G(\varphi,M)$) we have that $EU_\exists^2((\sigma^*)',\tau) >  EU_\exists^2(\sigma,\tau)$, contradicting the fact that $(\sigma,\tau)$ is an equilibrium for $G(\varphi,M)$. Analogously one can prove that $EU_\exists^1(\sigma,\tau^*) < EU_\exists^1(\sigma,\tau)$ cannot hold for any $\tau^*$.
\end{proof}
\begin{cor}
An $I\exists\exists$ occurring in a quantifier prefix $\vec Q$ is relevant if and only if it is $\epsilon$-relevant.
\end{cor} 

This means that if we have a structure $M$ and an IF formula $\psi$ such that $0 < v(\vec Q\psi,M),v( (\vec Q\psi)^{y\leftarrow x},M )<1$ (both formulas are undetermined) and $v(\vec Q\psi,M) \neq v( (\vec Q\psi)^{y\leftarrow x},M )$, then there are also $N,\chi$ such that $v(\vec Q\chi,N) \neq$ $\neq v( (\vec Q\chi)^{y\leftarrow x},N)$ but at least one out of and $\vec Q\chi$ and $(\vec Q\chi)^{y\leftarrow x}$ is determined.\\

Corollaries analogue to those of section \ref{SOMECON} do hold.

\section{Non-prenex quantifier sequences} \label{SYNTREE}
The results of the previous sections were proved only for prenex sequences of quantifiers. We show here that some generalization is possible. We will need these generalized results in the following section, where we shall study the expressive power of a fragment of IF logic which is not closed with respect to the usual prenex transformations.

\begin{df}
The \textbf{syntactical tree} of an IF formula $\varphi$ is the set of quantifiers, connectives, and atomic subformulas of $\varphi$, ordered according to superordination.
\end{df}
\begin{df}
An \textbf{positive initial syntactical tree} is a finite tree whose elements are either (occurrences of) conjunctions, disjunctions or quantifiers (with their slash sets), and which respects the following constraints:\\
1) each connective has exactly two successors \\
2) each quantifier has exactly one successor. 
\end{df}
The word \emph{positive} refers to the fact that we do not allow negation symbols to occur.
\begin{df}
A \textbf{dense open subset} $Y$ of a tree $T$ is $Y\subset T$ such that
\[
\forall y\in Y\forall t\in T(t\preceq y \rightarrow t\in Y)
\] 
\end{df}
\begin{df}
An IF formula $\varphi$ \textbf{begins with $T$} if $T$ is a tree and a dense open subset of the syntactical tree of $\varphi$.
\end{df}

\begin{df}
We say that a I$\exists\exists$ declaration of independence (say, independence of $y$ from $x$) occurs in a tree $T$ if $\{(\exists x/X),(\exists y/Y)\}\subset T$,  $(\exists x/X) \prec (\exists y/Y)$ and $x\in Y$. \\
We say it is \textbf{t-relevant} in $T$ if there exist a sentence $\varphi$ beginning with $T$, and a structure $M$, such that $\varphi$ and $\varphi^{y\leftarrow x}$ have different truth values on $M$. 
\end{df}

\begin{df}
A \textbf{(generalized) broken signalling sequence} occurring in a positive initial tree $T$ is a linearly ordered subset of $T$ consisting of quantifiers only and satisfying conditions 1-2-3-4 of definition \ref{DEFBSS}.
\end{df}

The reader who is acquainted with fenomena of signalling in IF logic may be surprised to see that disjunction symbols are not going to play a role in this generalized notion. The proof of the following result should clarify this point.

\begin{teo}  \label{NONPRENEX}
Suppose $T$ is a positive initial tree with an I$\exists\exists$ declaration of independence of $y$ from x. Then, the $I\exists\exists$ is t-relevant if and only if $T$ contains a generalized broken signalling sequence of the form $((\forall v_k/V_k),(\exists v_{k-1}/V_{k-1}),\dots,\\ (\exists v_1/V_1),(\exists x/X),(\exists y/Y))$.
\end{teo} 
\begin{proof}
$\Longleftarrow$) We define $\varphi$ by extending $T$ to a syntactical tree; this is achieved by extending each maximal chain (linearly ordered subset) $K$ of $T$ with the synctactical tree of some IF formula. If $\hat K$ is the chain which contains the I$\exists\exists$, then extend it with the syntactical tree of formula $(\exists w/W)(v_k=y \land x_1=w \land \dots \land x_n=w)$, $W$ being the set of all variables occurring in $\hat K$. For any other maximal chain $K$, let $c$ be the maximal element of $K\cap \hat K$. If $c$ is a disjunction, extend $K$ with the synctactical tree of a contradiction (e.g. $\exists u_k(u_k\neq u_k)$, where $u_k$ is a new variable); if $c$ is a conjunction, extend $K$ with a validity (e.g. $\exists u_h(u_h = u_h)$). \\
Then the proof follows the sames lines as the proof of the right-to-left part of Theorem \ref{THEONE} (keeping in mind that any choice at disjunctions which leads outside of $\hat K$ is a loss for $\eloise$, and any choice at conjunctions which leads outside of $\hat K$ is a loss for $\abelard$).\\
$\Longrightarrow$) The proof is completely analogous to that of the left-to-right part of theorem \ref{THEONE}. 
\end{proof}

\section{On the fragment of knowledge memory} \label{KMIF}
We described in section \ref{GTS} the game-theoretical notions of perfect recall, knowledge memory and action recall. These  turned out to be invariants of IF formulas, and so they determine some fragments of IF logic, for example:
\begin{df}
We denote as $\text{IF}_{\text{PR}}$ (resp. $\text{IF}_{\text{KM}}$, $\text{IF}_{\text{AR}}$) the fragment of regular IF formulas with perfect recall (resp. knowledge memory, action recall) for both players.
\end{df}
As we already mentioned, it was shown that, for what regards sentences and the notions of truth equivalence and falsity equivalence,  $\text{IF}_{\text{PR}}$ has the same expressive power as first-order logic:
\begin{teo}[Sevenster,\cite{Sev2006}]
Every  $\text{IF}_{\text{PR}}$ sentence $\varphi$ is truth equivalent (resp. falsity equivalent) with a first-order sentence.
\end{teo}
We show here that this result can be extended to the larger fragment  $\text{IF}_{\text{KM}}$. That is, lack of knowledge memory turns out to be essential in order to obtain greater expressive power than that of first-order logic. 
\begin{teo}
Every  $\text{IF}_{\text{KM}}$ sentence $\varphi$ is truth equivalent (resp. falsity equivalent) with a first-order sentence.
\end{teo}
\begin{proof}
Let $\varphi\in \text{IF}_{\text{KM}}$. Without loss of generality, we may suppose $\varphi$ to be in negation normal form, since the transformations needed for translating a formula into negation normal form preserve knowledge memory. While dealing with truth equivalence, we can also assume that the slash sets of universal quantifiers of $\varphi$ are empty. \\
Suppose first that $\varphi \in\text{IF}_{\text{AR}}$; then $\varphi\in\text{IF}_{\text{PR}}$, and the claim holds thanks to Sevenster's theorem.\\
Suppose instead $\varphi\notin\text{IF}_{\text{AR}}$. Let $(\exists x/X)$ and $(\exists y/Y)$ be quantifications occurring in $\varphi$ that are witness of the failure of action recall (i.e., $(\exists x/X)\prec(\exists y/Y)$ and $x\in Y$) and such that $(\exists y/Y)$ occurs with maximal depth in $\varphi$. Since $\varphi\in\text{IF}_{\text{KM}}$, it cannot contain any broken signalling sequence, in particular no broken signalling sequence ending in $(\exists x/X),(\exists y/Y)$. So, by Theorem \ref{NONPRENEX} $\varphi$ is truth equivalent to $\varphi^{y\leftarrow x}$ (here we used the fact that $\varphi$ is in negation normal form). And $\varphi^{y\leftarrow x}\in \text{IF}_{\text{KM}}$, thanks to the maximal depth of the occurrence of $(\exists y/Y)$ (suppose there is $(\exists z/Z) \succ (\exists y/(Y\setminus\{x\})$ such that $x \in Z$; then, the couple $(\exists x/X),(\exists z/Z)$ is a witness of the failure of action recall in $\varphi$, and $(\exists z/Z)$ occurs in $\varphi$ with greater depth than $(\exists y/Y)$ ).
So, we can iterate this elimination process until we remove all witnesses of failure of action recall, and obtain a formula $\varphi^0\in\text{IF}_{\text{PR}}$ which is truth equivalent to $\varphi$. But $\varphi^0$ is truth equivalent to some first-order formula due to Sevenster's theorem.\\
The case of falsity equivalence is proved using the dual of Theorem \ref{NONPRENEX}.
\end{proof}

\section{Final remarks}
We have given a characterization of the behaviour of existential declarations of independence in IF quantifier prefixes; we extended the result also to the more general framework of equilibrium semantics. We have formally proved that the $I\exists\exists$ are relevant precisely when they interrupt some flow of information from a universal quantification to an existential quantification (a flow which does not explicitly ''leak'' information - this is condition 4 in the definition of ''broken signalling sequence''). The characterization is syntactical and thus effective.
Dual results can be derived for $I\forall\forall$ declarations of independence.


Due to its effectiveness, the main theorem can be seen as a schema of equivalence rules
\[
\frac{\varphi}{\varphi^{y\leftarrow x}} (\text{provided that independence of $y$ from $x$ is irrelevant})
\]
and thus may be seen as a small contribution to the understanding of the proof theory of IF logic.
As a corollary, we obtained a game-theoretical proof of the fact that declarations of independence are meaningless in existential sentences, i.e., existential sentences are essentially first-order. Similar results were drawn for more general classes of IF sentences.

Through a mild generalization of the main theorem, we managed to prove that the knowledge memory fragment of IF logic has essentially first-order expressive power, thus strenghtening a previous result of Sevenster on the perfect recall fragment.

As a last observation, we repeat that our notion of relevance is weak: it manifests itself as existence of both a model and a sentence on which truth is influenced by removal of a declaration of independence. What happens if we fix a model and look at which sentences' truth values are influenced by removal of declarations of independence? Corollary \ref{MODELRELEV} provides an answer. We wonder, instead, what happens if we fix a formula and look at its possible models; it is our opinion that the problem is difficult to address, due to the complexity of signalling phenomena. Also, we have not investigated any notion of relevance for IF (open) formulas and their ''team semantics''.

\paragraph{Acknowledgements.}
I am grateful to Allen L.Mann, whose seminar at ESSLLI 2011 made me aware of the problem which is studied in this paper. He also provided helpful comments on some parts of the paper. Also, I am really indebted to Gabriel Sandu, who encouraged my work on IF logic and supported my visit to the University of Helsinki. He also read the paper and suggested some significant simplifications in proofs.

An intellectual debt is due to Theo M.V. Janssen, whose papers on IF logic were a source of inspiration for the present work.

\end{document}